\documentclass[12pt]{amsart}     
\usepackage{graphicx,a4wide,hyperref}

\newtheorem{theorem}{Theorem}[section]
\newtheorem{prop}[theorem]{Proposition}
\newtheorem{lemma}[theorem]{Lemma}
\newtheorem{corollary}[theorem]{Corollary}

\numberwithin{equation}{section} 

\begin{document}
\title{Redundancy of minimal weight expansions in Pisot bases}

\author[P. J. Grabner]{Peter J. Grabner\textsuperscript{\dag{}}}
\thanks{\textsuperscript{\dag{}}This author is supported by the Austrian
  Science Foundation FWF, project S9605, part of the Austrian National Research
  Network ``Analytic Combinatorics and Probabilistic Number Theory''.}
\address[P. Grabner]{Institut f\"ur Analysis und Computational Number Theory,
  Technische Universit\"at Graz, Steyrergasse 30, 8010 Graz, Austria}
\email{peter.grabner@tugraz.at}

\author[W. Steiner]{Wolfgang Steiner}
\address[W. Steiner]{LIAFA, CNRS, Universit\'e Paris Diderot -- Paris 7, Case
  7014, 75205 Paris Cedex 13, France}
\email{steiner@liafa.jussieu.fr}
\date{}

\begin{abstract}
  Motivated by multiplication algorithms based on redundant number
  representations, we study representations of an integer $n$ as a sum
  $n=\sum_k \varepsilon_k U_k$, where the digits $\varepsilon_k$ are taken from a
  finite alphabet $\Sigma$ and $(U_k)_k$ is a linear recurrent sequence
  of Pisot type with $U_0=1$. The most prominent example of a base sequence 
  $(U_k)_k$ is the sequence of Fibonacci numbers. We prove that the representations of minimal
  weight $\sum_k|\varepsilon_k|$ are recognised by a finite automaton and obtain
  an asymptotic formula for the average number of representations of minimal
  weight. Furthermore, we relate the maximal order of magnitude of the number
  of representations of a given integer to the joint spectral radius of a certain set of matrices.
\end{abstract}


\maketitle 
\section{Introduction}\label{sec:introduction}
Forming large multiples of elements of a given group plays an important role in
public key cryptosystems based on the Diffie-Hellman scheme (cf.\ for instance
\cite{Cohen_Frey_Avanzi+2006:handbook_of_elliptic}, especially
\cite{Doche2006:exponentiation}). In practice, the underlying groups are often
chosen to be the multiplicative group of a finite field $\mathbb{F}_q$ or the
group law of an elliptic curve (elliptic curve cryptosystems).

For $P$ an element of a given group (written additively), we need to form $nP$
for large $n\in\mathbb{N}$ in a short amount of time. One way to do this is the
\emph{binary method} (cf.\ \cite{Gathen_Gerhard1999:modern_computer_algebra}),
which is simply an applications of Horner's scheme to the binary expansion of
$n$. This method uses the operations of ``doubling'' and ``adding~$P$''. If we
write $n$ in its binary representation, the number of doublings is fixed by
$\lfloor\log_2n\rfloor$ and each \emph{one} in this representation corresponds
to an addition.  Thus the cost of the multiplication depends on the length of
the binary representation of $n$ and the number of ones in this representation.

In the case of the point group of an elliptic curve, addition and subtraction
are given by very similar expressions and are therefore equally costly. Thus it
makes sense to work with \emph{signed binary representations}, i.e., binary
representations with digits $\{0,\pm 1\}$. The advantage of these
representations is their redundancy: in general, $n$ has many different signed
binary representations. Then the number of non-zero digits in a signed binary
representation of $n$ is called the \emph{Hamming weight} of this
representation. Since each non-zero digit causes a group addition ($1$~causes
addition of~$P$, $-1$~causes subtraction of~$P$), one is interested in finding
a representation of $n$ having minimal Hamming weight. Such a minimal
representation was exhibited by Reitwiesner
\cite{Reitwiesner1960:binary_arithmetic}. The number of binary representations
of minimal weight has been analysed in
\cite{Grabner_Heuberger2006:number_optimal_base}.
    
In the present paper we propose to use Fibonacci-multiples instead of powers
of~$2$. The advantage of this choice is to avoid successive duplication (most
of the time), which uses a different formula in the case of the group law of an
elliptic curve. A~further advantage of these representations is the smaller
average weight compared to the binary representation (cf.\
\cite{Frougny_Steiner2008:minimal_weight_expansions}). More generally, we study
representations with linear recurrent base sequences of Pisot type. 
We calculate the number of representations of minimal weight with respect to these
numeration systems and obtain an asymptotic formula for the average number of
representations in the range $[-N,N]$.

A main tool of our study will be automata, which recognise the various
representations. As a general reference for automata in the context of number
representation we refer to
\cite{Frougny2002:numeration_systems,
Frougny_Sakarovitch2010:number_representations_automata}.
The books \cite{Lind_Marcus1995:introduction_symbolic_dynamics,
 Sakarovitch2009:elements_of_automata} provide the basic notions
of symbolic dynamics and automata theory.

\section{$U$-expansions and $\beta$-expansions of minimal weight}
\label{sec:beta-expans-minim}

\subsection{Setting}
Let $U = (U_k)_{k\ge0}$ be a strictly increasing sequence of integers with $U_0 = 1$, and $z = z_k z_{k-1} \cdots z_0$ a finite word on an alphabet $\Sigma \subseteq \mathbb{Z}$. 
We say that $z$ is a \emph{$U$-expansion} of the number $\sum_{j=0}^k z_j U_j$.
The \emph{greedy $U$-expansions} of positive integers~$n$, which are defined by
\begin{equation*}
n = \sum_{j=0}^k z_j U_j \quad \mbox{with} \quad \sum_{j=0}^\ell z_j U_j < U_{\ell+1} \quad \mbox{for}\ \ell = 0,1,\ldots,k,\ z_k \ne 0,
\end{equation*}
are well studied, in particular for the case when $U$ is the Fibonacci sequence $F = (F_k)_{k\ge0}$ with $F_0 = 1$, $F_1 = 1$, $F_k = F_{k-1} + F_{k-2}$ for $k \ge 2$, see e.g.\ \cite{Frougny2002:numeration_systems}.
The sum-of-digits function of greedy $U$-expansions with $U$ satisfying suitable linear recurrences has been studied by \cite{Petho_Tichy1989} and in several subsequent papers.

In the present paper we are interested in words with the smallest weight among all $U$-expansions of the same number. 
Here the \emph{weight} of $z$ is the absolute sum of digits $\|z\| = \sum_{j=0}^k |z_j|$.  
This weight is equal to the Hamming weight when $z \subseteq \{-1,0,1\}^*$, where $\Sigma^*$ denotes the set of finite words with letters in the alphabet~$\Sigma$.

We define the relation $\sim_U$ on words in $\mathbb{Z}^*$ by $z \sim_U y$ when $z$ and $y$ are $U$-expansions of the same number, i.e., 
\begin{equation*}
z_k z_{k-1} \cdots z_0 \sim_U y_\ell y_{\ell-1} \cdots y_0 \quad \mbox{if and only if} \quad
\sum_{j=0}^k z_j U_j = \sum_{j=0}^\ell y_j U_j.
\end{equation*}
Then the set of \emph{$U$-expansions of minimal weight} is 
\begin{equation*}
L_U = \{z \in \mathbb{Z}^*:\ \|z\| \le \|y\|\ \mbox{for all}\ y \in \mathbb{Z}^*\ \mbox{with}\ z \sim_U y\}.
\end{equation*}
Of course, leading zeros do not change the value and weight of a $U$-expansion.
In particular, every element of $0^* z$ is in $L_U$ if $z \in L_U$. 

Throughout the paper, we assume that there exists a \emph{Pisot number}~$\beta$, i.e., an algebraic integer $\beta > 1$ with $|\beta_i| < 1$ for every Galois conjugate $\beta_i \ne \beta$, such that $U$ satisfies (eventually) a linear recurrence with characteristic polynomial equal to the minimal polynomial of~$\beta$.
Then there exists some constant $c > 0$ such that 
\begin{equation} \label{eq:Gk}
U_k = c\, \beta^k + \mathcal{O}(|\beta_2|^k),
\end{equation}
where $\beta_2$ is the second largest conjugate of $\beta$ in modulus.

\subsection{Regularity of $L_U$}
For three particular sequences~$U$ (the Fibonacci sequence, the Tribonacci sequence and a sequence related to the smallest Pisot number), the set $L_U \cap \{-1,0,1\}^*$ is given explicitely in \cite{Frougny_Steiner2008:minimal_weight_expansions} by means of a finite automaton, see Figure~\ref{figexp} for the Fibonacci sequence.
Recall that an \emph{automaton} $\mathcal{A} = (Q,\Sigma,E,I,T)$ is a directed graph, where $Q$ is the set of vertices, traditionally called states, $I \subseteq Q$ is the set of initial states, $T \subseteq Q$ is the set of terminal states and $E \subseteq Q \times \Sigma \times Q$ is the set of edges (or transitions) which are labelled by elements of~$\Sigma$.
If $(p,a,q) \in E$, then we write $p \stackrel{a}{\to} q$.
A word in $\Sigma^*$ is \emph{accepted by~$\mathcal{A}$} if it is the label of a path starting in an initial state and ending in a terminal state.
The set of words which are accepted by~$\mathcal{A}$ is said to be \emph{recognised by~$\mathcal{A}$}.
A~\emph{regular language} is a set of words which is recognised by a finite automaton.
The main result of this subsection is the following theorem.

\begin{figure}[ht]
\includegraphics[scale=1.2]{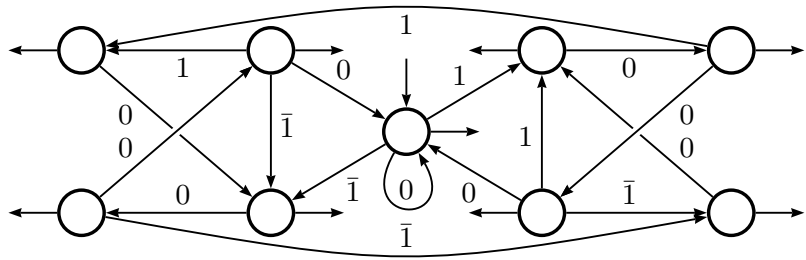}
\caption{Automaton recognising the set of $F$-expansions of minimal weight in $\{-1,0,1\}^*$.}
\label{figexp}
\end{figure}

\begin{theorem}\label{thm:ratU}
Let $U = (U_k)_{k\ge0}$ be a strictly increasing sequence of integers with $U_0 = 1$, satisfying eventually a linear recurrence with characteristic polynomial equal to the minimal polynomial of a Pisot number.
Then the set of $U$-expansions of minimal weight is recognised by a finite automaton.
\end{theorem}

First note that the structure of $L_U$ is similar to the structure of the $\beta$-expansions of minimal weight.
Here, $z = z_k z_{k-1} \cdots z_0 \in \mathbb{Z}^*$ is a \emph{$\beta$-expansion} of the number $\sum_{j=0}^k z_j \beta^j$.  
Similarly to~$\sim_U$, we define the relation $\sim_\beta$ on $\mathbb{Z}^*$ by
\begin{equation} \label{eq:simbeta}
z_k z_{k-1} \cdots z_0 \sim_\beta y_\ell y_{\ell-1} \cdots y_m \quad \mbox{if} \quad \sum_{j=0}^k z_j \beta^j = \sum_{j=m}^\ell y_j \beta^j.
\end{equation}
A~difference with $\sim_U$ is that $m$ can be chosen freely in~$\mathbb{Z}$; we have $z \sim_\beta y$ if~\eqref{eq:simbeta} holds for some $m \in \mathbb{Z}$.
The set of \emph{$\beta$-expansions of minimal weight} is 
\begin{equation*}
L_\beta = \{z \in \mathbb{Z}^*:\ \|z\| \le \|y\|\ \mbox{for all}\ y \in \mathbb{Z}^*\ \mbox{with}\ z \sim_\beta y\}.
\end{equation*}
(These definitions are equivalent to the ones in \cite{Frougny_Steiner2008:minimal_weight_expansions}.)
Now, leading and trailing zeros do not change the minimal weight property, i.e., $0^* L_\beta\, 0^* = L_\beta$. 
Theorem~3.11 in~\cite{Frougny_Steiner2008:minimal_weight_expansions} states that one can construct a finite automaton recognising~$L_\beta$.
The proof of the corresponding result for $L_U$ is slightly more complicated.
We start with the following proposition, which resembles Proposition~3.5 in \cite{Frougny_Steiner2008:minimal_weight_expansions}.

\begin{prop} \label{p:DG2}
Let $U$ be as in Theorem~\ref{thm:ratU}. 
Then there exists a positive integer $B$ such that
\begin{equation}\label{eq:DG2}
\forall\, k \ge 0,\ \exists\, b^{(k)} \in \mathbb{Z}^*:\ B\, 0^k \sim_U b^{(k)},\ \|b^{(k)}\| < B.
\end{equation}
\end{prop}

\begin{proof}
Let $U$ be a strictly increasing sequence of integers with $U_0 = 1$, $\beta$~a Pisot number of degree~$d$, and $h \ge 0$ an integer such that, for all $k \ge h+d$, $U_k$~is given by the linear recurrence with respect to the minimal polynomial of~$\beta$. 
By \cite[Proposition~3.5]{Frougny_Steiner2008:minimal_weight_expansions}, we know that for sufficiently large~$B$ there exists some $b = b_\ell \cdots b_m \in \mathbb{Z}^*$ such that $B = \sum_{j=m}^\ell b_j \beta^j$ and $\|b\| < B$.
Then we have $B\, 0^k \sim_U b_\ell \cdots b_m\, 0^{k+m}$ for all $k \ge h-m$.
For $0 \le k < h-m$, the weight of the greedy $U$-expansion of the integer $B\, U_k$ grows with $\mathcal{O}(\log B)$.
Therefore, there exists some positive integer $B$ satisfying~\eqref{eq:DG2}.
\end{proof}

\begin{prop} \label{p:LUB}
Let $U$ be as in Theorem~\ref{thm:ratU}.
If $B$ is a positive integer satisfying~\eqref{eq:DG2}, then $L_U \subseteq \{1-B,\ldots,B-1\}^*$.
If $B$ is a positive integer satisfying 
\begin{equation}\label{eq:DG}
\forall\, k \ge 0,\ \exists\, b^{(k)} \in \mathbb{Z}^*:\ B\, 0^k \sim_U b^{(k)},\ \|b^{(k)}\| \le B,
\end{equation}
then there exists for every $n \in \mathbb{Z}$ some $z \in L_U \cap \{1-B,\ldots,B-1\}^*$ with $z \sim_U n$.
\end{prop}

\begin{proof}
This can be proved similarly to Proposition~3.1 in \cite{Frougny_Steiner2008:minimal_weight_expansions}.
\end{proof}

For $U = F$, \eqref{eq:DG} holds with $B = 2$ since $2 \sim_F 10$, $20 \sim_F 4 \sim_F 101$ and $2\,0^k \sim_F 1001\,0^{k-2}$ for $k \ge 2$. 
Therefore, we are mainly interested in the language $L_F \cap \{-1,0,1\}^*$, which is recognised by the automaton in Figure~\ref{figexp} \cite[Theorem~4.7]{Frougny_Steiner2008:minimal_weight_expansions}.
The minimal positive integer satisfying~\eqref{eq:DG2} is $B = 3$.
Here, we have $3 \sim_F 100$, $30 \sim_F 6 \sim_F 1001$ and $3\,0^k \sim_F 10001\,0^{k-2}$ for $k \ge 2$, whereas $20$ is clearly a $F$-expansion of minimal weight.

Next we show the following generalisation of a well known result for $\beta$-expansions \cite[Corollary~3.4]{Frougny1992:representation_numbers_finite}.
For a subclass of sequences~$U$, this result can be found in \cite{Frougny1989,Frougny1992:representation_numbers_finite}.

\begin{prop} \label{p:ratZ}
Let $U$ be as in Theorem~\ref{thm:ratU}.
Then, for every finite alphabet $\Sigma \subset \mathbb{Z}$,
\begin{equation*}
Z_{U,\Sigma} = \{z \in \Sigma^*:\ z \sim_U 0\}
\end{equation*} 
is recognised by a finite automaton.
\end{prop}

\begin{proof}
Let $U$ be as in the proof of Proposition~\ref{p:DG2}.
Let $\beta = \beta_1, \beta_2, \ldots, \beta_d$ be the conjugates of~$\beta$.
Then there exist constants $c_i \in \mathbb{Q}(\beta_i)$ such that
\begin{equation} \label{eq:Gbetai}
U_{h+k} = \sum_{i=1}^d c_i \beta_i^k \quad \mbox{for all} \quad k \ge 0.
\end{equation}

For any word $z_k \cdots z_0 \in \Sigma^*$ we can write 
\begin{equation} \label{eq:zm}
\sum_{j=0}^{k-h} z_{h+j} \beta^j = \sum_{j=0}^{d-1} m_j \beta^j \quad \mbox{with} \quad m_{d-1} \cdots m_0 \in \mathbb{Z}^d.
\end{equation}
We have $z_k \cdots z_0 \sim_U m_{d-1} \cdots m_0\, z_{h-1} \cdots z_0$ (with $z_j = 0$ for $k \le j < h$), thus 
\begin{equation} \label{eq:zG0}
z_k \cdots z_0 \sim_U 0 \quad \mbox{if and only if} \quad \sum_{j=0}^{d-1} m_j U_{h+j} + \sum_{j=0}^{h-1} z_j U_j = 0.
\end{equation}
By~\eqref{eq:zm}, we obtain
\begin{equation} \label{eq:conjbounded}
\bigg|\sum_{j=0}^{d-1} m_j \beta_i^j\bigg| \le \frac{\max_{a\in\Sigma}|a|}{1-|\beta_i|} \quad \mbox{for} \quad i = 1,2,\ldots,d-1.
\end{equation}
By~\eqref{eq:zG0}, \eqref{eq:Gbetai} and~\eqref{eq:conjbounded}, we obtain that $|\sum_{j=0}^{d-1} m_j \beta^j|$ is bounded as well if $z_k \cdots z_0 \sim_U 0$.
There are only finitely many words $m_{d-1} \cdots m_0 \in \mathbb{Z}^d$ such that all conjugates of $\sum_{j=0}^{d-1} m_j \beta^j$ are bounded.
Therefore, there are only finitely many possibilities for $m_{d-1} \cdots m_0 \in \mathbb{Z}^d$, $z_{h-1} \cdots z_0 \in \Sigma^*$ such that $m_{d-1} \cdots m_0\, z_{h-1} \cdots z_0 \sim_U 0$.
Set
\begin{equation*}
T_U(z_{h-1} \cdots z_0) = \bigg\{\sum_{j=0}^{d-1} m_j \beta_{h+j} + \sum_{j=0}^{h-1} z_j \beta_j\ \Big|\ m_{d-1} \cdots m_0\, z_{h-1} \cdots z_0 \sim_U 0\bigg\}
\end{equation*}
and $M = \max\bigcup_{z'\in\Sigma^h} T_U(z')$.

Let $\mathcal{A}_{U,\Sigma}$ be the automaton with initial state $(0, 0^h)$ and transitions $(s, z_{h-1} \cdots z_0) \stackrel{a}{\rightarrow} (\beta s - a, z_{h-2} \cdots z_0 a)$, $a \in \Sigma$, such that $|\beta s - a| < M + \max_{b\in\Sigma}|b|/(\beta-1)$.  
A~state $(s, z') \in \mathbb{Z}[\beta] \times \Sigma^h$ is terminal if and only if $s \in T_U(z')$.
Then $\mathcal{A}_{U,\Sigma}$ is finite and recognises~$Z_{U,\Sigma}$.
\end{proof}

Now, we can prove Theorem~\ref{thm:ratU}.
As in \cite{Frougny_Steiner2008:minimal_weight_expansions}, we make use of \emph{letter-to-letter transducers}, which are automata with transitions labelled by pairs of digits.
If $(z_k,y_k) \cdots (z_0,y_0)$ is the sequence of labels of a path from an initial to a terminal state, we say that the transducer accepts the pair of words $(z,y)$, with $z = z_k \cdots z_0$ being the input and $y = y_k \cdots y_0$ being the output of the transducer.

\begin{proof}[Proof of Theorem~\ref{thm:ratU}]
By Propositions~\ref{p:DG2} and~\ref{p:LUB}, there exists a positive integer $B$ such that $L_U \subseteq \Sigma^*$ with $\Sigma = \{1-B, \ldots, B-1\}$. 

In the proof of Theorem~3.10 in \cite{Frougny_Steiner2008:minimal_weight_expansions}, it was shown that there exists a finite letter-to-letter transducer $\mathcal{T}$ with the following property:
For every word $z \in (\Sigma L_\beta \cap L_\beta \Sigma) \setminus L_\beta$, i.e., $z \in \Sigma^* \setminus L_\beta$ and every proper factor of $z$ is in~$L_\beta$, there exist integers $\ell, m$ and a word $y \in \Sigma^*$ such that $(0^\ell z\, 0^m, y)$ is the label of a path in~$\mathcal{T}$ leading from $(0,0)$ to $(0,\delta)$, with $\delta < 0$.
The transitions are of the form $(s, \delta) \stackrel{(a,b)}\longrightarrow (\beta s + b - a, \delta + |b| - |a|)$, $a, b \in \Sigma$.
This means that $y \sim_\beta z$ and $\|y\| < \|z\|$.
Since $\mathcal{T}$ is finite, we can choose $m \le K$ for some constant~$K$.
By the assumptions on~$U$, we obtain that $z\, 0^k \sim_U y\, 0^{k-m}$ for all $k \ge h+K$, thus $z\, 0^k \not\in L_U$.
Note that $z\, 0^k \not\in L_U$ implies that $z' z\, z'' \not\in L_U$ for all $z' \in \Sigma^*$, $z'' \in \Sigma^k$.
Now, since $\Sigma^* \setminus L_\beta$ is recognised by a finite automaton, we also have an automaton recognising the set of words $z = z_k \cdots z_0 \in \Sigma^* \setminus L_U$ with $z_k \cdots z_{h+K} \not\in L_\beta$.

It remains to consider the words $z = z_k \cdots z_0 \in \Sigma^* \setminus L_U$ with $z_k \cdots z_{h+K} \in L_\beta$ (if $k \ge h+K$).
Let $y = y_\ell \cdots y_0 \sim_U z$ with $y \in L_U$, and assume w.l.o.g.\ $\ell \ge k$.
All these pairs of words $(0^{\ell-k}z,y)$ are accepted by a letter-to-letter transducer~$\mathcal{T}'$ with $(0,0^h,0^h,0)$ as initial state, transitions
\begin{equation*}
(s,\, z_{h-1} \cdots z_0,\, y_{h-1} \cdots y_0,\, \delta) \stackrel{(a,b)}\longrightarrow (\beta s + b - a,\, z_{h-2} \cdots z_0 a,\, y_{h-2} \cdots y_0 b,\, \delta + |b| - |a|),
\end{equation*}
$a, b \in \Sigma$, and terminal states $(s,z',y',\delta)$ such that $s \in T_U(y') - T_U(z')$, $\delta < 0$.
We show that $\mathcal{T}'$ is a finite transducer.
As in the proof of Proposition~\ref{p:ratZ}, we obtain states $(s,z',y',\delta)$ with $s$ in a finite subset of~$\mathbb{Z}[\beta]$, more precisely $|s| < 2M + 2(B-1)/(\beta-1)$ and the conjugate of~$s$ corresponding to~$\beta_i$ is bounded by $2(B-1)/(1-|\beta_i|)$ for $2 \le i \le d$.
Clearly, there are only finitely many possibilities for $z', y' \in \Sigma^h$.
By the previous paragraph, $y \in L_U$ implies $y_\ell \cdots y_{h+K} \in L_\beta$.
As in the proof of Theorem~3.10 in \cite{Frougny_Steiner2008:minimal_weight_expansions}, for $m \ge h+K$, a large difference $\delta = \|y_k \cdots y_m\| - \|z_k \cdots z_m\|$ contradicts the assumption that $z_k \cdots z_m \in L_\beta$ and $y_k \cdots y_m \in L_\beta$.
Since $h+K$ and $\Sigma$ are finite, the difference between $\|z_k \cdots z_m\|$ and $\|y_k \cdots y_m\|$ is bounded for $0 \le m < h+K$ as well, thus $\mathcal{T}'$ is finite.

If we modify $\mathcal{T}'$ by adding those states to the set of initial states which can be reached from $(0,0^h,0^h,0)$ by a path with input consisting only of zeros, then the input automaton of the modified transducer recognises a subset of $\Sigma^* \setminus L_U$ containing all words $z_k \cdots z_0 \in \Sigma^* \setminus L_U$ with $z_k \cdots z_{h+K} \in L_\beta$.
Therefore, $\Sigma^* \setminus L_U$ is regular as the union of two regular languages, and the complement $L_U$ is regular as well.
\end{proof}

\subsection{Properties of the automata}
The \emph{trim minimal automaton} recognising a set~$H$ is the deterministic automaton with minimal number of states recognising~$H$, where \emph{deterministic} means that there is a unique initial state and from every state there is at most one transition labelled by $a$ for every $a \in \Sigma$.
Let $\mathcal{M}_{U,\Sigma}$ and $\mathcal{M}_{\beta,\Sigma}$ be the trim minimal automata recognising $L_U \cap \Sigma^*$ and $L_\beta \cap \Sigma^*$ respectively; let $A_{U,\Sigma}$ and $A_{\beta,\Sigma}$ be the respective adjacency matrices.
We will see that the automata $\mathcal{M}_{U,\Sigma}$ and $\mathcal{M}_{\beta,\Sigma}$ are closely related.
We show first that the matrix $A_{\beta,\Sigma}$ is primitive, using the following lemma.

\begin{lemma} \label{l:zerosbounded}
Let $\mathcal{T}$ be a finite letter-to-letter transducer with transitions of the form $(s, \delta) \stackrel{(a,b)}\longrightarrow (\beta s + b - a, \delta + |b| - |a|)$, $\beta \ne 0$, $a, b \in \mathbb{Z}$.
Then the number of consecutive zeros in the input of a path in~$\mathcal{T}$ not running through a state of the form $(0,\delta)$ is bounded.
\end{lemma}

\begin{proof}
Let $(0^k,y)$ be the label of a path starting from $(s,\delta)$ with $s \ne 0$.
Then the path leads to a state $(s', \delta+\|y\|)$, thus $\|y\|$ is bounded by the finiteness of~$\mathcal{T}$.
If $y$ starts with~$0^j$, then the path leads to $(\beta^j s, \delta)$, thus the finiteness of $\mathcal{T}$ implies that $j$ is bounded.
If the path avoids states $(s,\delta)$ with $s = 0$, then the boundedness of $\|y\|$ and the boundedness of consecutive zeros in $y$ imply that $k$, which is the length of $y$, is bounded.
\end{proof}

\begin{prop} \label{p:Aprimitive}
Let $\beta$ be a Pisot number and $0 \in \Sigma \subseteq \mathbb{Z}$.
Then $A_{\beta,\Sigma}$ is primitive.
\end{prop}

\begin{proof}
We show that, from every state in $\mathcal{M}_{\beta,\Sigma}$, the path labelled by $0^k$ leads to the initial state if $k$ is sufficiently large.

First note that $z \in L_\beta$ implies $z\, 0^k \in L_\beta$ for all $k \ge 0$, thus there always exists a path labelled by~$0^k$.
Suppose that this path does not lead to the initial state from some state.
Then there exist words $z, z' \in L_\beta \cap \Sigma^*$ with $z\, 0^k\, z' \not\in L_\beta$.
We can assume w.l.o.g.\ $z\, 0^k\, z' \in (\Sigma L_\beta \cap L_\beta \Sigma) \setminus L_\beta$.
As in the proof of Theorem~\ref{thm:ratU}, there exist integers $\ell, m$ and a word $y \in \Sigma^*$ such that $(0^\ell z\, 0^m, y)$ is the label of a path in~$\mathcal{T}$ leading from $(0,0)$ to $(0,\delta)$, with $\delta < 0$. 
If this path ran through a state $(0,\delta)$ while reading the input $0^k$ between $z$ and~$z'$, then the corresponding prefix of~$y$ would be a word $y' \sim_\beta z$ and the corresponding suffix of~$y$ would be a word $y'' \sim_\beta z'$. 
Since $\|y'\| + \|y''\| = \|y\| < \|z\| + \|z'\|$, we had $\|y'\| < \|z\|$ or $\|y''\| < \|z'\|$, contradicting that $z, z' \in L_\beta$.
Therefore, Lemma~\ref{l:zerosbounded} yields that $k$ is bounded.

Hence, for sufficiently large~$k$, $0^k$ is a synchronizing word of $\mathcal{M}_{\beta,\Sigma}$ leading to the initial state.
Since $\mathcal{M}_{\beta,\Sigma}$ was assumed to be a trim minimal automaton, this implies that $\mathcal{M}_{\beta,\Sigma}$ is strongly connected, thus $A_{\beta,\Sigma}$ is irreducible.
Now, the primitivity of $A_{\beta,\Sigma}$ follows from the fact that there is a loop labelled by $0$ in the initial state. 
\end{proof}

\begin{prop} \label{p:Gequalbeta}
Let $U$ be as in Theorem~\ref{thm:ratU} and $0 \in \Sigma \subseteq \mathbb{Z}$.
Then the automaton $\mathcal{M}_{U,\Sigma}$ has a unique strongly connected component. Up to the set of terminal states, this component is equal to~$\mathcal{M}_{\beta,\Sigma}$.
\end{prop}

\begin{proof}
We first show that every word $z \in L_\beta$ is the label of a path starting in the initial state of~$\mathcal{M}_{U,\Sigma}$.
Suppose that $z\, 0^k \not\in L_U$ for some large $k \ge 0$, then there exists an integer~$\ell$ and a word $y \in \Sigma^*$ such that $(0^\ell z\,0^k, y)$ is accepted by the finite transducer $\mathcal{T}'$ in the proof of Theorem~\ref{thm:ratU}.
As in Lemma~\ref{l:zerosbounded}, we obtain that the path must run through a state $(0,z',y',\delta)$ while reading~$0^k$, which implies that $y \sim_\beta z$.
Moreover, we have $\delta < 0$, thus $\|y\| < \|z\|$, contradicting that $z \in L_\beta$.
This shows that the directed graph $\mathcal{M}_{U,\Sigma}$ contains the directed graph~$\mathcal{M}_{\beta,\Sigma}$.

Now, consider an arbitary word $z \in L_U$ such that the corresponding path ends in a strongly connected component of~$\mathcal{M}_{U,\Sigma}$.
This means that we have $z\, z' \in L_U$ for arbitarily long words~$z'$.
Since $z\, z' \in L_U$ implies $z\, 0^k \in L_U$, where $k$ is the length of~$z'$, we obtain that $z \in L_\beta$.
Therefore, the strongly connected components of $\mathcal{M}_{U,\Sigma}$ are contained in~$\mathcal{M}_{\beta,\Sigma}$.
Since $\mathcal{M}_{\beta,\Sigma}$ has a unique strongly connected component by Proposition~\ref{p:Aprimitive}, the same holds for~$\mathcal{M}_{U,\Sigma}$.
\end{proof}

In Section~\ref{sec:aver-numb-repr}, we also use that the difference between the length of the longest $U$-expansion of minimal weight (without leading zeros) and e.g.\ the greedy $U$-expansion is bounded.

\begin{lemma} \label{l:K}
Let $U$ be as in Theorem~\ref{thm:ratU} and $\Sigma$ a finite subset of~$\mathbb{Z}$. 
Let $z = z_k \cdots z_0 \in \Sigma^*$ with $z_k \ne 0$, $y = y_\ell \cdots y_0 \in L_U$ with $y_\ell \ne 0$.
There exists a constant $m \ge 0$ such that $z \sim_U y$ implies $\ell \le k+m$.
\end{lemma}

\begin{proof}
For $\ell < k$, the assertion is trivially true. 
If $\ell \ge k$, then $(0^{\ell-k} z,y)$ is accepted by a transducer similar to $\mathcal{T}'$ in the proof of Theorem~\ref{thm:ratU}, with states $(s,z',y',\delta)$ such that $s$ is in a finite set.
Now $\delta$ can be unbounded.
However, $y \in L_U$ implies that the path labelled by $(0^{\ell-k}, y_\ell \cdots y_{k+1})$ starting from $(0,0^h,0^h,0)$ runs through states $(s,z',y',\delta)$ with bounded~$\delta$, cf.\ the proof of Theorem~3.10 in \cite{Frougny_Steiner2008:minimal_weight_expansions}.
As in Lemma~\ref{l:zerosbounded}, we obtain that $\ell-k$ is bounded.
\end{proof}

\section{Average number of representations}
\label{sec:aver-numb-repr}
In this section we study the function $f(n)$ counting the number of different
$U$-expansions of minimal weight (without leading zeros) of the integer~$n$ in~$\Sigma^*$, with $\{0,1\} \subseteq \Sigma \subseteq \mathbb{Z}$.
As in Theorem~\ref{thm:ratU}, $U = (U_k)_{k\ge0}$ is assumed to be a strictly increasing sequence of integers with $U_0 = 1$, satisfying eventually a linear recurrence with characteristic polynomial equal to the minimal polynomial of a Pisot number~$\beta$.
We will give precise asymptotic information about the average number of representations
$\frac1{2N-1} \sum_{|n|<N} f(n)$. As a general reference for the study of the asymptotic
behaviour of digital functions we refer to
\cite{Drmota_Grabner2010:analysis_of_digital}. In order to exhibit the
fluctuating main term of this sum we introduce a measure $\mu$ on $\big[\frac{\min\Sigma}{\beta-1}, \frac{\max\Sigma}{\beta-1}\big]$. The
construction of this measure is similar to the distribution measures of
infinite Bernoulli convolutions as studied in
\cite{Erdos1939:family_symmetric_bernoulli}. There it encodes the number of
representations of integers as sums of Fibonacci numbers.

As in Section~\ref{sec:beta-expans-minim}, let $\mathcal{M}_{U,\Sigma}$ be the trim minimal automaton recognising $L_U \cap \Sigma^*$.
Denote by $A_{U,a}$ the adjacency matrix of all transitions in $\mathcal{M}_{U,\Sigma}$ labelled by the digit~$a$. 
The total adjacency matrix of the automaton is then $A_{U,\Sigma} = \sum_{a\in\Sigma} A_{U,a}$.
Let $\mathcal{M}_{\beta,\Sigma}, A_{\beta,a}, A_{\beta,\Sigma}$ be the corresponding objects for $\beta$-expansions.

Let $f_k(n)$ denote the number of words $z \in L_U \cap \Sigma^k$ with $z \sim_U n$, i.e., the number of $U$-expansions of minimal weight of length $k$ of an integer~$n$. 
If $|n| < U_k$, then the length of the greedy $U$-expansion of $|n|$ is at most~$k$.
Then, by Lemma~\ref{l:K}, there exists a constant $m \ge 0$ such that every $U$-expansions of minimal weight without leading zeros is of length at most~$k+m$.
By adding leading zeros, we obtain that $f_j(n) = f(n)$ for all $j \ge k+m$. 

We define a sequence of measures by
\begin{equation}\label{eq:massn}
\mu_k = \frac1{M_k} \sum_{n\in\mathbb{Z}} f_k(n)\, \delta_{\frac{n}{U_k}},
\end{equation}
where $\delta_x$ denotes the unit point mass concentrated in $x$ and
\begin{equation*}
M_k = \sum_{n\in\mathbb{Z}} f_k(n) = \#(L_U \cap \Sigma^k).
\end{equation*}
We notice that all points $\frac{n}{U_k}$ with $f_k(n)>0$ lie in the interval $\big(\sum_{j=0}^{k-1} \frac{U_j}{U_k}\big) [\min\Sigma, \max\Sigma]$.

As a first step of reduction we replace the measure $\mu_k$ by the measure $\nu_k$ given by
\begin{equation*} 
\nu_k = \frac1{M_k} \sum_{z \in L_U\cap\Sigma^k} \delta_{g(z)} \quad \mbox{with} \quad g(z_{k-1} \cdots z_0) = \sum_{j=0}^{k-1} z_j \beta^{j-k}.
\end{equation*}
By~\eqref{eq:Gk}, we have
\begin{equation*}
\sum_{j=0}^{k-1} z_j \frac{U_j}{U_k} - g(z_{k-1} \cdots z_0) = \mathcal{O}(\beta^{-k}),
\end{equation*}
thus
\begin{equation}\label{eq:mu-nu}
|\widehat{\mu}_k(t) - \widehat{\nu}_k(t)| = \mathcal{O}(|t|\, \beta^{-k}).
\end{equation}
From this it follows that $(\mu_k)_k$ and $(\nu_k)_k$ tend to the
same limiting measure~$\mu$.

In order to compute the characteristic function of $\nu_k$ we consider the
weighted adjacency matrix of $\mathcal{M}_{U,\Sigma}$,
\begin{equation*} 
A_{U,\Sigma}(t) = \sum_{z\in\Sigma} e(z t)\, A_{U,z},
\end{equation*}
where we use the notation $e(t) = e^{2\pi it}$. 
Then we have
\begin{equation*}
\widehat{\nu}_k(t) = \frac1{M_k} \sum_{z \in L_U\cap\Sigma^k} e\big(g(z) t\big) =
\frac1{M_k} \mathbf{v}_1 A_{U,\Sigma}\big(e(t\beta^{-1})\big) A_{U,\Sigma}\big(e(t\beta^{-2})\big) \cdots A_{U,\Sigma}\big(e(t\beta^{-k})\big) \mathbf{v}_2,
\end{equation*}
where $\mathbf{v}_1$ is the indicator (row) vector of the initial state of~$\mathcal{M}_{U,\Sigma}$ and $\mathbf{v}_2$ is the indicator (column) vector of the terminal states of~$\mathcal{M}_{U,\Sigma}$. 

\begin{lemma} \label{lem17}
The adjacency matrix $A_{U,\Sigma} = A_{U,\Sigma}(0)$ of the automaton $\mathcal{M}_{U,\Sigma}$ has a unique dominating eigenvalue~$\alpha$, which is positive and of multiplicity~$1$.
\end{lemma}

\begin{proof}
By Proposition~\ref{p:Gequalbeta}, every non-zero eigenvalue of $A_{U,\Sigma}$ is an eigenvalue of~$A_{\beta,\Sigma}$, with the same multiplicity.
Therefore, the lemma follows from Proposition~\ref{p:Aprimitive} and the Perron-Frobenius theorem.
\end{proof}

By Lemma~\ref{lem17}, there exists a positive constant $C$ such that 
\begin{equation}\label{eq:gesamtmasse_n}
M_k = \mathbf{v}_1 A_{U,\Sigma}^k \mathbf{v}_2 = C \alpha^k  + \mathcal{O}\big((|\alpha_2|+\varepsilon)^k\big)
\end{equation}
for every $\varepsilon>0$, where $\alpha$ and $\alpha_2$ are the largest and second
largest roots of the characteristic polynomial of $A_{U,\Sigma}$.

\begin{lemma}\label{lem-norm}
  Let $A$ be a $n \times n$-matrix with complex entries. There exists a matrix
  norm $\|\cdot\|$ satisfying $\|A\| = \rho(A)$ (the spectral radius) if and
  only if for all eigenvalues $\lambda$ of $A$ with $|\lambda| = \rho(A)$ the
  algebraic and geometric multiplicity are equal.
\end{lemma}

\begin{proof}
Assume that for all $\lambda$ with $|\lambda| = \rho(A)$ the algebraic and
geometric multiplicities are equal. Then there exists a non-singular matrix~$S$,
such that
\begin{equation*}
SAS^{-1}=
\begin{pmatrix}
\lambda_1&0&0&0&0&\ldots&0\\
0&\lambda_2&0&0&0&\ldots&0\\
0&0&\ddots&0&\vdots&\ddots&\vdots\\
0&0&\ldots&\lambda_r&0&\ldots&0\\
0&0&\ldots&0&\\
\vdots&\vdots&\ddots&\vdots&&B\\
0&0&\ldots&0
\end{pmatrix},
\end{equation*}
with $|\lambda_1|=\cdots=|\lambda_r|=\rho(A)$ and $\rho(B)<\rho(A)$. Then by
\cite[Lemma~5.6.10 and Theorem~5.6.26]{Horn_Johnson1985:matrix_analysis} there
is a norm $\|\cdot\|_{n-r}$ on $\mathbb{C}^{n-r}$ such that the induced norm on
matrices satisfies $\|B\|<\rho(A)$. Define the norm on $\mathbb{C}^n$ by
\begin{equation*}
\|\mathbf{x}\|=\|\mathrm{pr}_1S\mathbf{x}\|_r+
\|\mathrm{pr}_2S\mathbf{x}\|_{n-r},
\end{equation*}
where $\mathrm{pr}_1$ denotes the projection to the first $r$ coordinates and
$\mathrm{pr}_2$ the projection to the $n-r$ last coordinates; $\|\cdot\|_r$ is
just the $\ell^1$-norm on $\mathbb{C}^r$. Then we have
\begin{equation*}
\frac{\|A\mathbf{x}\|}{\|\mathbf{x}\|}=
\frac{\rho(A)\|\mathrm{pr}_1S\mathbf{x}\|_r+\|B\mathrm{pr}_2S\mathbf{x}\|_{n-r}}
{\|\mathrm{pr}_1S\mathbf{x}\|_r+\|\mathrm{pr}_2S\mathbf{x}\|_{n-r}}\leq\rho(A)
\end{equation*}
and therefore $\|A\|=\rho(A)$ by the fact that $\|A\|\geq\rho(A)$ for all
norms. Here we have used
$\mathrm{pr}_2SAS^{-1}=B\mathrm{pr}_2$.

If on the other hand $A$ is not diagonalisable for some $\lambda$ with
$|\lambda| = \rho(A)$, then there exist two vectors $\mathbf{e}_1$ and
$\mathbf{e}_2$ such that
\begin{equation*}
A\mathbf{e}_1=\lambda \mathbf{e}_1+\mathbf{e}_2\quad\text{and }
A\mathbf{e_2}=\lambda\mathbf{e}_2.
\end{equation*}
Then we have
\begin{equation*}
A^k\mathbf{e}_1=\lambda^k\mathbf{e}_1+k\lambda^{k-1}\mathbf{e}_2.
\end{equation*}
Let $\|\cdot\|$ be any norm on $\mathbf{C}^n$. Then 
\begin{equation*}
\|\lambda^{-k}A^k\mathbf{e}_1\|=\|\mathbf{e}_1+k\lambda^{-1}\mathbf{e}_2\|\geq
k|\lambda|^{-1}\|\mathbf{e}_1\|-\|\mathbf{e}_2\|
\end{equation*}
shows that $\|\lambda^{-k}A^k\mathbf{e}_1\|$ is unbounded, whereas
$\|A\|=\rho(A)$ would imply that this sequence is bounded by
$\|\mathbf{e}_1\|$. Thus there is no induced matrix norm with
$\|A\|=\rho(A)$. Since by
\cite[Theorem~5.6.26]{Horn_Johnson1985:matrix_analysis} for every norm there is
an induced norm, which is smaller, there cannot exist a matrix norm with
$\|A\|=\rho(A)$.
\end{proof}

By Lemma~\ref{lem-norm} there exists a norm on $\mathbb{C}^{\#\,\mathrm{states}\,\mathrm{of}\,\mathcal{M}_{U,\Sigma}}$ such that the induced norm on matrices satisfies $\|A_{U,\Sigma}(0)\| = \rho(A_{U,\Sigma}(0)) = \alpha$.  From now on we use this norm. 
By differentiability of the entries of $A_{U,\Sigma}(t)$ and the fact that the norm
$\|\cdot\|$ is comparable to the $\ell^1$-norm, there exists a positive
constant $C$ such that
\begin{equation*}
\big\|A_{U,\Sigma}(t) - A_{U,\Sigma}(0)\big\| \leq C|t|.
\end{equation*}

We will prove that $(\nu_k)_k$ (and therefore $(\mu_k)_k$) weakly tends to a limit measure by showing that $(\widehat{\nu}_k(t))_k$ tends to a limit $\widehat{\nu}(t) = \widehat{\mu}(t)$.

\begin{lemma}\label{lem3}
The sequence of measures $(\mu_k)_k$ defined by \eqref{eq:massn} converges weakly to a probability measure~$\mu$. 
The characteristic functions satisfy the inequality
\begin{equation}\label{eq:muhat-ineq}
\big|\widehat{\mu}_k(t) - \widehat{\mu}(t)\big| =
\begin{cases}
\mathcal{O}\big(|t|\,\beta^{-\eta k}\big)&\text{for }|t|\leq1,\\
\mathcal{O}\big(|t|^\eta \beta^{-\eta k}\big)&\text{for }|t|\geq1,
\end{cases}
\end{equation}
with
\begin{equation}\label{eq:eta}
\eta=\frac{\log\alpha-\log(|\alpha_2|+\varepsilon)}
{\log\beta+\log\alpha-\log(|\alpha_2|+\varepsilon)}
\end{equation}
for any $\varepsilon>0$.  
The constants implied by the $\mathcal{O}$-symbol depend only on~$\varepsilon$.
\end{lemma}

\begin{proof}
We study the product
\begin{equation*}
P_k(t) = \alpha^{-k} \prod_{j=1}^k A(t\beta^{-j}),
\end{equation*}
with $A = A_{U,\Sigma}$.
For $|t|\leq1$ we estimate
\begin{multline*}
\big\|P_k(t) - P_k(0)\big\| = \alpha^{-k}\, \bigg\|\prod_{j=1}^k \Big(A(0)+\big(A(t\beta^{-j})-A(0)\big)\Big) - A(0)^k\bigg\|\\
\leq \alpha^{-k} \sum_{\ell=1}^k \big\|A(0)\big\|^{k-\ell} \hspace{-1.7em} 
\sum_{1\leq j_1<j_2<\cdots<j_\ell\leq k} \hspace{-1.7em} \big\|A(t\beta^{-j_1})-A(0)\big\| \cdot \big\|A(t\beta^{-j_2})-A(0)\big\| \cdots \big\|A(t\beta^{-j_\ell})-A(0)\big\|\\
\leq \sum_{\ell=1}^{k} \alpha^{-\ell} \sum_{1\leq j_1<j_2<\cdots<j_\ell\leq k} C^\ell |t|^\ell \beta^{-(j_1+\cdots+j_\ell)}\\
\leq \sum_{\ell=1}^{k} \frac1{\ell!} \alpha^{-\ell} C^\ell |t|^\ell \bigg(\sum_{j=1}^k \beta^{-j}\bigg)^\ell \leq \exp\bigg(\frac{C|t|}{\alpha(\beta-1)}\bigg) - 1 = \mathcal{O}(|t|).
\end{multline*}

Furthermore, we have for $j > k > \ell$ and $1 \leq |t| \leq \beta^{\ell}$
\begin{multline*}
\big\|P_k(t) - P_j(t)\big\| = \big\|P_{k-\ell}(t\beta^{-\ell}) P_\ell(t) - P_{j-\ell}(t\beta^{-\ell}) P_\ell(t)\big\| \\
\leq \big\|P_\ell(t)\big\| \Big(\big\|P_{k-\ell}(t\beta^{-\ell}) - P_{k-\ell}(0)\big\| + \big\|P_{j-\ell}(t\beta^{-\ell}) - P_{j-\ell}(0)\big\| + \big\|P_{k-\ell}(0) - P_{j-\ell}(0)\big\|\Big) \\
= \mathcal{O}\big(|t|\beta^{-\ell}\big) + \mathcal{O}\bigg(\bigg(\frac{|\alpha_2|+\varepsilon}\alpha\bigg)^{k-\ell}\,\bigg) = \mathcal{O}\big(|t|^\eta \beta^{-\eta k}\big).
\end{multline*}
Here we have used the fact that $\|P_\ell(t)\|$ is uniformly bounded for all
$\ell\in\mathbb{N}$ and all $t\in\mathbb{R}$, since all entries of $P_\ell(t)$ are bounded by
the entries of $\alpha^{-\ell}A(0)^\ell$ and the entries of this matrix
converge.  In the last step we have set $\ell = \lceil(1-\eta)\log_\beta|t|+\eta
k\rceil$.  The inequality is valid for $j>k>\log_\beta|t|$.

We now assume that $|t|\leq1$ and $j>k>\ell$. Then we have
\begin{multline*}
\big|\widehat{\nu}_k(t) - \widehat{\nu}_j(t)\big|  = \bigg|\frac{\alpha^k}{M_k} \mathbf{v}_1 P_k(t) \mathbf{v}_2 - \frac{\alpha^j}{M_j} \mathbf{v}_1 P_j(t) \mathbf{v}_2\bigg| \\
= \bigg|\frac{\alpha^k}{M_k} \mathbf{v}_1 P_{k-\ell}(t\beta^{-\ell}) P_\ell(t) \mathbf{v}_2 - \frac{\alpha^j}{M_j} \mathbf{v}_1 P_{j-\ell}(t\beta^{-\ell}) P_\ell(t) \mathbf{v}_2\bigg| \\
\leq \bigg|\frac{\alpha^k}{M_k} \mathbf{v}_1 P_{k-\ell}(0) P_\ell(t) \mathbf{v}_2 - \frac{\alpha^j}{M_j} \mathbf{v}_1 P_{j-\ell}(0 )P_\ell(t) \mathbf{v}_2\bigg| + \mathcal{O}\big(|t|\beta^{-\ell}\big) \\
= \bigg|\frac{\alpha^k}{M_k} \mathbf{v}_1 P_{k-\ell}(0) \big(P_\ell(t) - P_\ell(0)\big) \mathbf{v}_2 - \frac{\alpha^j}{M_j} \mathbf{v}_1 P_{j-\ell}(0) \big(P_\ell(t) - P_\ell(0)\big) \mathbf{v}_2\bigg| + \mathcal{O}\big(|t|\beta^{-\ell}\big) \\
= |t|\, \mathcal{O}\bigg(\beta^{-\ell} + 
\bigg(\frac{|\alpha_2|+\varepsilon}\alpha\bigg)^{k-\ell}\,\bigg),
\end{multline*}
where we have used $\frac{\alpha^k}{M_k} \mathbf{v}_1 P_k(0) \mathbf{v}_2 = 1$ in
the fourth line. Setting $\ell = \lfloor\eta k\rfloor$ gives
\begin{equation*}
\big|\widehat{\nu}_k(t) - \widehat{\nu}_j(t)\big| = \mathcal{O}\big(|t|\beta^{-\eta k}\big).
\end{equation*}

Thus $\widehat{\nu}_k(t)$ converges uniformly on compact subsets of $\mathbb{R}$ to a
continuous limit $\widehat{\mu}(t)$, and the measures $\nu_k$ tend to a measure $\mu$
weakly. From this together with \eqref{eq:mu-nu} the two inequalities
\eqref{eq:muhat-ineq} are immediate.
\end{proof}

\begin{lemma} \label{l:upper}
There exists a positive real number $\gamma < \alpha$ such that 
\begin{equation*}
\textstyle\max_{x\in\mathbb{Z}[\beta]} \#\big\{z_{k-1} \cdots z_0 \in L_\beta \cap \Sigma^k:\ \sum_{j=0}^{k-1} z_j \beta^j = x\big\}\, =\, \mathcal{O}(\gamma^k).
\end{equation*}
\end{lemma}

\begin{proof}
Similarly to~\eqref{eq:gesamtmasse_n}, we have 
\begin{equation*}
\#(L_\beta \cap \Sigma^k) = \mathbf{v}'_1\, A_{\beta,\Sigma}^k\, \mathbf{v}'_2 = \mathcal{O}(\alpha^k),
\end{equation*}
where $\mathbf{v}'_1$ is the indicator (row) vector of the initial state of~$\mathcal{M}_{\beta,\Sigma}$ and $\mathbf{v}'_2 = (1,\ldots,1)^T$ is the indicator (column) vector of the terminal states of~$\mathcal{M}_{\beta,\Sigma}$. 
We show that there exists some $\ell \ge 1$ and a matrix $\tilde{A}$ with $\tilde{A} < A_{\beta,\Sigma}^\ell$ (entrywise) such that
\begin{equation} \label{e:Atilde}
\textstyle\#\big\{z_{k-1} \cdots z_0 \in L_\beta \cap \Sigma^k:\ \sum_{j=0}^{k-1} z_j \beta^j = x\big\} \le \mathbf{v}'_1\, \tilde{A}^{\lfloor k/\ell\rfloor}\, A_{\beta,\Sigma}^{k-\lfloor k/\ell\rfloor\ell}\, \mathbf{v}'_2
\end{equation}
for all $x \in \mathbb{Z}[\beta]$, $k \ge 0$.

Each entry in $A_{\beta,\Sigma}^\ell$ counts the number of paths of length~$\ell$ in $\mathcal{M}_{\beta,\Sigma}$ between two states $q$ and~$q'$.
By the proof of Proposition~\ref{p:Aprimitive}, there exists $k_1 \ge 0$ such that the path labelled by $0^{k_1}$ leads from every state to the initial state.
Let $k_2 \ge 0$ be such that $1\, 0^{k_2}$ leads from the initial state to itself, $k_3$~be the maximal distance of a state from the initial state, and $\ell = k_1 + k_2 + k_3 + 1$.
Then, for any two states $q, q'$, there exists a $z' \in \Sigma^{k_3}$ such that paths labelled by $0^{k_1}\, 1\, 0^{k_2}\, z'$ and by $0^{k_1+1+k_2}\, z'$ (of length~$\ell$) run from $q$ to~$q'$.
It is well known that the words $(z_{k-1} \cdots z_0, y_{k-1} \cdots y_0)$ with $\sum_{j=0}^{k-1} z_j \beta^j = \sum_{j=0}^{k-1} y_j \beta^j$ are recognised by a finite automaton with transitions of the form $s \stackrel{(a,b)}\to \beta s + b - a$, see e.g.\ \cite{Frougny_Steiner2008:minimal_weight_expansions}.
For sufficiently large $k_1$ and~$k_2$, there is no path labelled by $(0^{k_1}\, 1\, 0^{k_2}, 0^{k_1+1+k_2})$ in this automaton.
Therefore, for any fixed $x \in \mathbb{Z}[\beta]$, any word $z_{k-1} \cdots z_j$, $\ell \le j \le k$, leading to the state~$q$ in~$\mathcal{M}_{\beta,\Sigma}$ cannot be prolonged by all labels of paths of length~$\ell$ between $q$ and~$q'$ when we want to obtain a word $z_{k-1} \cdots z_0 \in L_\beta \cap \Sigma^k$ with $\sum_{j=0}^{k-1} z_j \beta^j = x$.

This means that, for sufficiently large~$\ell$, \eqref{e:Atilde} holds with $\tilde{A}$ taken as the matrix with every entry being one smaller than that of~$A_{\beta,\Sigma}^\ell$.
Let $\tilde{\alpha}$ be the dominant eigenvalue of~$\tilde{A}$, then $\tilde{\alpha} < \alpha^\ell$ and the lemma holds with $\gamma = \tilde{\alpha}^{1/\ell}$.
\end{proof}

\begin{corollary}
The counting function $f$ satisfies $f(n) = \mathcal{O}(|n|^{\log_\beta\gamma})$ for some $\gamma < \alpha$.
\end{corollary}

\begin{prop}\label{prop-modulus}
Let $\gamma$ be as in Lemma~\ref{l:upper}.
Then the measure $\mu$ satisfies
\begin{equation}\label{eq:meas-dim}
\mu\big([x,y]\big) = \mathcal{O}\big((y-x)^\theta\big)
\end{equation}
with
\begin{equation}\label{eq:theta}
\theta=\frac{\log\alpha - \log\gamma}{\log\beta}.
\end{equation}
\end{prop}

\begin{proof}
Let $x < y$, and $\ell = \lfloor-\log_\beta(y-x)\rfloor$.
Recall that
\begin{equation}\label{eq:self-sim}
\mu\big([x,y]\big) = \lim_{k\to\infty} \frac1{M_k} \sum_{z \in L_U\cap\Sigma^k:\, g(z) \in [x,y]} \delta_{g(z)}.
\end{equation}
Let $z = z_{k-1} \cdots z_0 \in L_U \cap \Sigma^k$.
By Proposition~\ref{p:Gequalbeta}, we have $z_{k-1} \cdots z_{k-\ell} \in L_\beta$ for sufficiently large~$k$.
If $g(z) \in [x,y]$, then 
\begin{equation*}
\sum_{j=0}^{\ell-1} z_{j+k-\ell} \beta^j \in \beta^\ell [x,y] - \bigg[\frac{\min\Sigma}{\beta-1}, \frac{\max\Sigma}{\beta-1}\bigg].
\end{equation*}
Since $y-x \le \beta^{-\ell}$, this implies that $\sum_{j=0}^{\ell-1} z_{j+k-\ell} \beta^j$ lies in an interval of bounded size. 
For all conjugates $\beta_i \ne \beta$, we have $\big|\sum_{j=0}^{\ell-1} z_{j+k-\ell} \beta_i^j\big| \le \max_{a\in\Sigma}|a|/(1-|\beta_i|)$, thus $\sum_{j=0}^{\ell-1} z_{j+k-\ell} \beta^j$ can take only a bounded number of values in~$\mathbb{Z}[\beta]$. (The bound does not depend on the choice of $[x,y]$.)
Then $z_{k-\ell-1} \cdots z_0 \in L_U \cap \Sigma^{k-\ell}$ and Lemma~\ref{l:upper} yield that
\begin{equation*}
\mu\big([x,y]\big) = \mathcal{O}\bigg(\gamma^\ell \lim_{k\to\infty}\frac{M_{k-\ell}}{M_k}\bigg) = \mathcal{O}\bigg(\bigg(\frac{\gamma}{\alpha}\bigg)^\ell\,\bigg).
\end{equation*}
Combining this with $\ell = -\log_\beta(y-x) + \mathcal{O}(1)$ gives~\eqref{eq:meas-dim}.
\end{proof}

We use Proposition~\ref{prop-modulus} and the following lemma to establish purity of the measure~$\mu$.

\begin{lemma}
[{\cite[Theorem~35]{Jessen_Wintner1935:distribution_functions_riemann}},
{\cite[Lemma~1.22~(ii)]{Elliott1979:probabilistic_number_theory_I}}]
\label{lem-jessen-wintner}
Let $Q=\prod_{k=0}^\infty Q_k$ be an infinite product of discrete spaces
equipped with a measure $\kappa$, which satisfies Kolmogorov's $0$-$1$-law
(i.e., every tail event has either measure $0$ or $1$). Furthermore,
let $X_k$ be a sequence of random variables defined on the spaces $Q_k$, such
that the series $X=\sum_{k=0}^\infty X_k$ converges $\kappa$-almost everywhere.
Then the distribution of $X$ is either purely discrete, or purely singular
continuous, or absolutely continuous with respect to Lebesgue measure.
\end{lemma}

\begin{prop}\label{prop-pure}
The measure $\mu$ is pure, i.e., it is either absolutely continuous or
purely singular continuous.
\end{prop}

\begin{proof}
We equip the shift space
\begin{equation*}
\mathcal{K} = \big\{(z_k)_{k\ge0}:\ z_k z_{k-1} \cdots z_0 \in L_\beta \cap \Sigma^* \ \mbox{for all}\ k \ge 0\big\}
\end{equation*}
associated to the automaton $\mathcal{M}_{\beta,\Sigma}$ with the measure
\begin{align*}
\kappa\big([z_0,z_1,\ldots,z_{\ell-1}]\big) & = \lim_{k\to\infty} \frac1{M_k} \#\big\{y_{k-1} \cdots y_0 \in L_\beta:\ y_{\ell-1} \cdots y_0 = z_{\ell-1} \cdots z_0\big\} \\
& = \lim_{k\to\infty} \frac1{\mathbf{v}'_1A_{\beta,\Sigma}^k\mathbf{v}'_2}
\mathbf{v}'_1 A_{\beta,\Sigma}^{k-\ell} A_{\beta,z_{\ell-1}} \cdots A_{\beta,z_0} \mathbf{v}_2
\end{align*}
given on the cylinder set
\begin{equation*}
[z_0,z_1,\ldots,z_{\ell-1}] = \big\{(y_k)_{k\ge0} \in \mathcal{K}:\ y_{\ell-1} \cdots y_0 = z_{\ell-1} \cdots z_0\big\}.
\end{equation*}
Then $\kappa$ can be written in terms of the transition matrices
\begin{equation*}
\kappa\big([z_0,z_1,\ldots,z_{\ell-1}]\big) = \frac1{\mathbf{v}\mathbf{v}'_2} \alpha^{-\ell}
\mathbf{v} A_{\beta,z_{\ell-1}} \cdots A_{\beta,z_0} \mathbf{v}'_2,
\end{equation*}
where $\mathbf{v}$ is the left Perron-Frobenius eigenvector of the matrix~$A_{\beta,\Sigma}$.
Let $\mathbf{w}$ denote the right Perron-Frobenius eigenvalue of the matrix $A_{\beta,\Sigma}$ with $\mathbf{v}\mathbf{w} = 1$. 
Then by positivity of all entries of $\mathbf{w}$ and $\mathbf{v}'_2$ the measure $\tilde{\kappa}$ given by
\begin{equation*}
\tilde{\kappa}\big([z_0,z_1,\ldots,z_{\ell-1}]\big) = \alpha^{-\ell} \mathbf{v} A_{\beta,z_{\ell-1}} \cdots A_{\beta,z_0} \mathbf{w}
\end{equation*}
is equivalent to~$\kappa$.

The measure $\tilde\kappa$ is strongly mixing and therefore ergodic with
respect to the shift. Thus $\tilde\kappa$ and $\kappa$ satisfy the hypotheses
of Lemma~\ref{lem-jessen-wintner}.

The continuity of $\mu$ is an immediate consequence of
Proposition~\ref{prop-modulus}.
\end{proof}

In order to give an error bound for the rate of convergence of the measures
$\mu_k$ to the measure~$\mu$, we will use the following version of the
Berry-Esseen inequality, which was proved in
\cite{Grabner1997:functional_iterations_stopping}.
\begin{prop}\label{prop1}
  Let $\mu_1$ and $\mu_2$ be two probability measures with
  their Fourier transforms defined by
  \begin{equation*}
\widehat{\mu}_k(t)=\int_{-\infty}^\infty e^{2\pi itx}\,d\mu_k(x),\quad
  k=1,2.
\end{equation*}
  Suppose that $\big(\widehat{\mu}_1(t)-\widehat{\mu}_2(t)\big)\, t^{-1}$ is
  integrable on a neighbourhood of zero and $\mu_2$ satisfies
  \begin{equation*}
\mu\big((x,y)\big)\leq c\,|x-y|^\theta
\end{equation*}
  for some $0<\theta<1$. Then the
  following inequality holds for all real $x$ and all $T>0$:
  \begin{multline*}
    \Big|\mu_1\big((-\infty,x)\big) - \mu_2\big((-\infty,x)\big)\Big| \leq
\Bigg|\int\limits_{-T}^T\hat J(T^{-1}t)\,
      (2\pi it)^{-1}\, \big(\widehat{\mu}_1(t)-\widehat{\mu}_2(t)\big)\,
    e^{-2\pi ixt}\,dt\Bigg|\\
  +\bigg(c+\frac1{\pi^2}\bigg)\, T^{-\frac{2\theta}{2+\theta}} + \Bigg|
    \frac1{2T}\int\limits_{-T}^T \bigg(1-\frac{|t|}T\bigg)
    \big(\widehat{\mu}_1(t)-\widehat{\mu}_2(t)\big)\, e^{-2\pi ixt}\,dt\Bigg|,
\end{multline*}
  where
  \begin{equation*}
 \hat J(t)=\pi t(1-|t|)\cot\pi t+|t|.
\end{equation*}
\end{prop}

\begin{lemma}\label{lem4}
The measures $\mu_k$ satisfy
\begin{equation}\label{eq:berry}
\big|\mu_k\big((x,y)\big) - \mu\big((x,y)\big)\big| = \mathcal{O}\big(\beta^{-\zeta k}\big)
\end{equation}
uniformly for all $x,y\in\mathbb{R}$ with
$\zeta=\frac{2\theta\eta}{\eta(\theta+2)+2\theta}$.
\end{lemma}

\begin{proof}
We apply Proposition~\ref{prop1} to the measures $\mu_k$ and~$\mu$. For this
purpose we use the inequalities \eqref{eq:muhat-ineq} to obtain
\begin{multline*}
\big|\mu_k\big((-\infty,x)\big) - \mu\big((-\infty,x)\big)\big| = \mathcal{O}\Bigg(\beta^{-\eta k} \int\limits_{-1}^1\,dt\Bigg) + \mathcal{O}\Bigg(\beta^{-\eta k} \hspace{-.5em} \int\limits_{1\leq|t|\leq T} \hspace{-.5em} |t|^{\eta-1}\,dt\Bigg) + \mathcal{O}\Big(T^{{-\frac{2\theta}{2+\theta}}}\Big)\\
+ \mathcal{O}\Bigg(\beta^{-\eta k} \frac1T \int\limits_{-1}^1|t|\,dt\Bigg) + \mathcal{O}\Bigg(\beta^{-\eta k} \frac1T \hspace{-.5em} \int\limits_{1\leq|t|\leq T} \hspace{-.5em} |t|^{\eta}\,dt\Bigg) = \mathcal{O}\big(\beta^{-\zeta n}\big)
\end{multline*}
by choosing $T=\beta^{\zeta\frac{2+\theta}{2\theta} k}$.
\end{proof}

Now the statement of the asymptotic behaviour of the average
$\frac1{2N-1}\sum_{|n|<N}f(n)$ is a consequence of the preceding discussion of
the properties of~$\mu$. Combining Lemma~\ref{lem3},
Proposition~\ref{prop-modulus}, and Lemma~\ref{lem4} we obtain the following
theorem.
\begin{theorem}\label{thm12}
  The summatory function of the number of representations of $n$ with minimal
  weight satisfies
\begin{equation}\label{eq:sumf1}
\sum_{|n|<N} f(n) = N^{\log_\beta\alpha}\, \Phi(\log_\beta N) + \mathcal{O}(N^{\lambda}),
\end{equation}
where $\Phi$ denotes a continuous periodic function of period $1$ and
\begin{equation*}
\lambda=\frac{\log\alpha}{\log\beta}-
\frac{2\theta\eta}{\eta(\theta+2)+2\theta},
\end{equation*}
$\eta$ given by \eqref{eq:eta}, and $\theta$ given by \eqref{eq:theta}.
\end{theorem}
\begin{proof}
  Using the definition of $\mu_k$ in~\eqref{eq:massn} and the value $m$ given
  by Lemma~\ref{l:K} we have
\begin{equation*}
\sum_{|n|<N}f(n)=M_k\,\mu_k\big((-N/U_k,N/U_k)\big),
\end{equation*}
where we choose $k=\lfloor\log_\beta N\rfloor+m$. Replacing $\mu_k$ by~$\mu$,
using $U_k=C\beta^k+\mathcal{O}(|\beta_2|^k)$,
$M_k=D\alpha^k+\mathcal{O}((|\alpha_2|+\varepsilon)^k)$ and taking all error terms
into account yields
\begin{multline*}
  \sum_{|n|<N}f(n) = D\, \alpha^{\lfloor\log_\beta N\rfloor+m}\, \mu\big(\big(-\beta^{\log_\beta N-\lfloor\log_\beta N\rfloor-m}/C,\, \beta^{\log_\beta N-\lfloor\log_\beta N\rfloor-m}/C\big)\big)\\
  +\mathcal{O}\big(\alpha^k\beta^{-\zeta k}\big)+
  \mathcal{O}\bigg(\alpha^k\bigg(\frac{|\beta_2|}\beta\bigg)^{\theta k}\,\bigg)+
  \mathcal{O}\big((|\alpha_2|+\varepsilon)^k\big),
\end{multline*}
with $\zeta$ as in Lemma~\ref{lem4}.
Defining 
\begin{equation*}
\Phi(t) = D\, \alpha^{m+\lfloor t\rfloor-t}\, \mu\big(\big(-\beta^{t-\lfloor t\rfloor-m}/C,\, \beta^{t-\lfloor t\rfloor-m}/C\big)\big)
\end{equation*}
for $t\geq0$ yields~\eqref{eq:sumf1}. The periodicity of $\Phi$ follows from
the definition. The continuity of $\Phi$ in non-integer points follows from the
continuity of~$\mu$. The fact that $\Phi(0)=\lim_{t\to1-}\Phi(t)$ is a
consequence of the self-similarity of the measure $\mu$
\begin{equation*}
\mu(\beta^{-1}A)=\alpha^{-1}\mu(A)\text{ for }A\subset[-\beta^{-m},\beta^{-m}];
\end{equation*}
this follows from \eqref{eq:self-sim} by the observation that multiplication by
$\beta^{-1}$ corresponds to adding a prefix $0$ in the representation and
reading a leading zero does not change the state of the automaton.
\end{proof}
\section{Exact number of representations}

To obtain the exact number of $U$-expansions of minimal weight representing an
integer~$n$, we choose one of these expansions and look at the transducer which
transforms any other expansion into it.  One way of choosing a output language
of such a transducer is to take the set of greedy $U$-expansions (for positive
numbers) and their symmetric counterparts (to represent negative numbers).
A~possible deficiency of this language is that it is not contained in~$L_U$.
In the following, we describe a regular language $G_U \subseteq L_U$ such that,
for every $n \in \mathbb{Z}$, there is, up to leading zeros, a unique word $z
\in G_U$ with $z \sim_U n$.

We call a word $z = z_k \cdots z_0 \in L_U$ \emph{greedy $U$-expansion of
  minimal weight} if
\begin{equation*}
0^{\min(\ell-k,0)}\, |z_k| \cdots |z_0| \ge 0^{\min(k-\ell,0)}\, |y_\ell| \cdots |y_0| \quad \mbox{for all}\ y = y_\ell \cdots y_0 \in L_U\ \mbox{with}\ z \sim_U y,
\end{equation*}
with respect to the lexicographical order.
The set of all greedy $U$-expansions of minimal is denoted by~$G_U$. 

\begin{lemma}
  Let $U$ be as in Theorem~\ref{thm:ratU}.  For any $n \in \mathbb{Z}$, there
  is, up to leading zeros, a unique word $z \in G_U$ with $z \sim_U n$.
\end{lemma}

\begin{proof}
  Let $n \in \mathbb{Z}$.  First note that by Lemma~\ref{l:K} there are, up to
  leading zeros, only finitely many words $z \in G_U$ with $z \sim_U n$.
  Therefore there exists a greedy $U$-expansion of minimal weight $z = z_k
  \cdots z_0$ with $z \sim_U n$.

  Let $y = y_\ell \cdots y_0$ be another word in $G_U$ with $y \sim_U n$.  Then
  we have $0^{\min(\ell-k,0)} |z_k| \cdots |z_0| = 0^{\min(k-\ell,0)}
  |y_\ell| \cdots |y_0|$.  Neglecting leading zeros, we can assume w.l.o.g.\
  that $k = \ell$.  We have $\frac{y_k+z_k}2 \cdots \frac{y_0+z_0}2 \in
  \mathbb{Z}^*$ because $\frac{y_j+z_j}2 = z_j$ in case $y_j = z_j$,
  $\frac{y_j+z_j}2 = 0$ in case $y_j = -z_j$; and $\frac{y_k+z_k}2 \cdots
  \frac{y_0+z_0}2 \sim_U n$.  If we had $y_j \neq z_j$ for some~$j$, then
  $\frac{y_k+z_k}2 \cdots \frac{y_0+z_0}2$ would have smaller weight than~$z$,
  contradicting $z \in L_U$.  Therefore, $z$ and $y$ differ only by leading
  zeros.
\end{proof}

\begin{theorem}\label{thm:greedyU}
Let $U$ be as in Theorem~\ref{thm:ratU}.
Then $G_U$ is recognised by a finite automaton.
\end{theorem}

\begin{proof}
  It suffices to show that $L_U \setminus G_U$ is a regular language.  Since
  the complement of a regular language is regular, see e.g.\
  \cite[Lemma~3.9]{Frougny_Steiner2008:minimal_weight_expansions}, this implies
  that $G_U$ is regular.

  Let $z = z_k \cdots z_0 \in L_U \setminus G_U$.  Then there exists a $y =
  y_\ell \cdots y_0 \in L_U$ with $y \sim_U z$ and $0^{\min(\ell-k,0)}\, |z_k|
  \cdots |z_0| < 0^{\min(k-\ell,0)}\, |y_\ell| \cdots |y_0|$.  Since $L_U$ is a
  regular language, the product $\{(z_k \cdots z_0, y_k \cdots y_0) \mid z_k
  \cdots z_0, y_k \cdots y_0 \in L_U,\, k \ge 0\}$ is recognised by a finite
  automaton.  One can construct a finite automaton recognising the
  lexicographic relation $|z_k| \cdots |z_0| < |y_k| \cdots |y_0|$.  By
  Proposition~\ref{p:ratZ}, the same holds for $z_k \cdots z_0 \sim_U y_k
  \cdots y_0$.  Therefore, there exists a finite automaton which recognises the
  set of pairs $(z_k \cdots z_0, y_k \cdots y_0) \in L_U \times L_U$ such that
  $z_k \cdots z_0 \sim_U y_k \cdots y_0$ and $|z_k| \cdots |z_0| < |y_k| \cdots
  |y_0|$.  Let $H$ be the projection of this set to the first coordinate, then
  $H$ is regular.  Moreover, $L_U \setminus G_U$ is the set of words which are
  obtained from words in $H$ by removing or appending initial zeros.  This
  language is regular as well.  By the remarks at the beginning of the proof,
  this proves the lemma.
\end{proof}

\begin{prop} \label{p:conv} Let $U$ be as in Theorem~\ref{thm:ratU}.  Then
  there exists a finite letter-to-letter transducer recognising the pairs $(z_k
  \cdots z_0, y_k \cdots y_0) \in L_U \times G_U$ with $z_k \cdots z_0 \sim_U
  y_k \cdots y_0$.
\end{prop}

\begin{proof}
  Similarly to the proof of Theorem~\ref{thm:greedyU}, this follows from the
  regularity of $L_U$ and~$G_U$, and from Proposition~\ref{p:ratZ}.
\end{proof}

Note that in Proposition~\ref{p:conv} the input and output must have the same
length.  In particular, the input must be as least as long as the corresponding
word in $G_U$ without leading zeros.  By Lemma~\ref{l:K}, the difference
between any two words in $L_U$ without leading zeros with the same value as
$U$-expansions is bounded.  Therefore, there exists a \emph{transducer with
  initial function} which computes, for every word in $L_U$ without leading
zeros, the corresponding word in $G_U$ without leading zeros, see
\cite{Frougny1992:representation_numbers_finite}.

We have the following corollary of Proposition~\ref{p:conv}, where
$\mathcal{N}_{U,\Sigma}$ denotes the trim minimal letter-to-letter transducer
of Proposition~\ref{p:conv}.

\begin{corollary} \label{c:transducer} Let $U, \Sigma, f$ be as in
  Section~\ref{sec:aver-numb-repr}, $n \in \mathbb{Z}$ and $y \in G_U$ with $y
  \sim_U n$.  Then $f(n)$ is given by the number of successful paths in
  $\mathcal{N}_{U,\Sigma}$ with $y$ as output.
\end{corollary}

It was shown in \cite{Frougny_Steiner2008:minimal_weight_expansions} that the
transducer in Figure~\ref{f:normal} transforms any $F$-expansion of minimal
weight, after possible addition of a leading zero, into the corresponding
$F$-expansion of minimal weight avoiding the factors $11$, $1\bar1$, $101$,
$10\bar1$, $1001$ and their opposites.  Here, we write $\bar{1}$ instead
of~$1$, and opposite means that $1$'s and $\bar{1}$'s are exchanged.  
This transducer shows that the set of outputs is~$G_F$, thus this transducer is equal to $\mathcal{N}_{F,\{-1,0,1\}}$.

\begin{figure}
\includegraphics{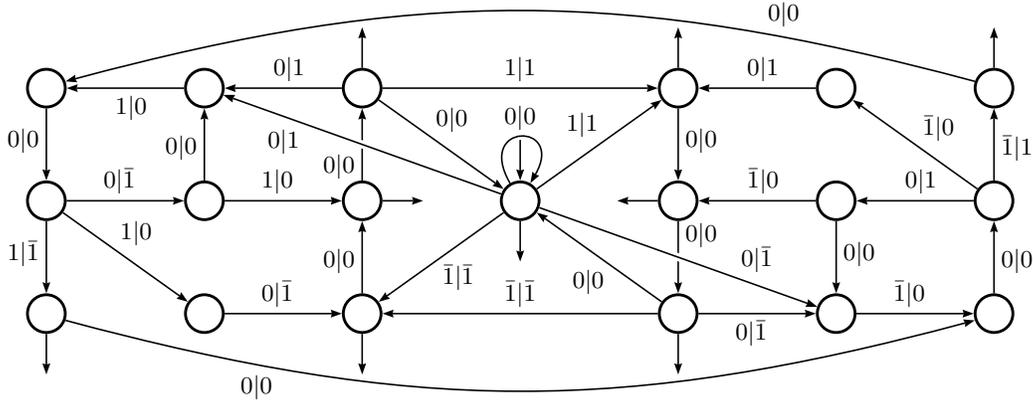}
\caption{Transducer $\mathcal{N}_{F,\{-1,0,1\}}$ normalising $F$-expansions of
  minimal weight in $\{-1,0,1\}^*$. Here $a\,|\,b$ stands for $(a,b)$ in the transition labels.} \label{f:normal}
\end{figure}
 
Now we can relate the growth of $f(n)$ to the \emph{joint spectral radius} of
the set $\{R_{U,a}:\, a \in \Sigma\}$, where $R_{U,a}$ denotes the adjacency
matrix of the transitions with output $a$ in~$\mathcal{N}_{U,\Sigma}$.  The
joint spectral radius is defined by
\begin{equation*}
\rho\big(\{R_{U,a}:\, a \in \Sigma\}\big) = \lim_{k\to\infty} \max\big\{
\|R_{U,z_{k-1}} \cdots R_{U,z_0}\|^{1/k} \mid z_{k-1} \cdots z_0 \in
\Sigma^k\big\},
\end{equation*}
where $\|\cdot\|$ is any matrix norm.  The definition of the joint spectral radius is due to
\cite{Rota_Strang1960}; an overview of its properties and its calculation can
be found in \cite{Jungers2009}.

\begin{theorem} \label{t:jointspectralradius} Let $U, \Sigma, f$ be as in
  Section~\ref{sec:aver-numb-repr}.  For any $\varepsilon > 0$, we have $f(n) =
  \mathcal{O}(|n|^{\log_\beta(\gamma+\varepsilon)})$, where $\gamma$ is the joint
  spectral radius of the set of matrices $\{R_{U,a}:\, a \in \Sigma\}$.
\end{theorem}

\begin{proof}
  For any non-zero $n \in \mathbb{Z}$, the length of the word $y \in G_U$ with
  $y \sim_U n$ and without leading zeros is $\log_\beta |n| + \mathcal{O}(1)$.
  By Corollary~\ref{c:transducer}, the number of successful paths
  in~$\mathcal{N}_{U,\Sigma}$ with output $y_{k-1} \cdots y_0$ is $\mathbf{v}
  R_{U,y_{k-1}} \cdots R_{U,y_0} \mathbf{w}$ for vectors $\mathbf{v},
  \mathbf{w}$ corresponding to the initial and the terminal states
  of~$\mathcal{N}_{U,\Sigma}$, thus $f(n) =
  \mathcal{O}((\gamma+\varepsilon)^{\log_\beta |n|})$.
\end{proof}

For $U = F$ and $\Sigma = \{-1,0,1\}$, we have an explicit formula for the
maximal number of elements of $L_F \cap \{-1,0,1\}^k$ with the same value.

\begin{theorem} 
Let $U = F$ and $\Sigma = \{-1,0,1\}$.
For every $k \ge 1$, we have
\begin{equation*}
\max_{n\in\mathbb{Z}} f_k(n) = 2^{\lfloor(k-1)/3\rfloor}.
\end{equation*}
\end{theorem}

\begin{proof}
First we show that 
\begin{equation} \label{e:maxvRw}
\max_{y\in\Sigma^k} \mathbf{v} R_{F,y} \mathbf{w} = 2^{\lfloor(k-1)/3\rfloor},
\end{equation}
where $R_{F,y_{k-1} \cdots y_0} = R_{F,y_{k-1}} \cdots R_{F,y_0}$, and $\mathbf{v}, \mathbf{w}$ are as in the proof of Theorem~\ref{t:jointspectralradius}.
From the structure of~$G_F$, it is clear that $R_{F,y} = 0$ if $y$ contains a factor $11, 1\bar{1}, 101, 10\bar{1}, 1001$ or its opposite. 
We have the following entrywise relations between matrices:
\begin{gather*}
R_{F,10^7} \le R_{F,100\bar{1}0000},\ R_{F,1000\bar{1}0} \le R_{F,100\bar{1}0},\ R_{F,10^k\bar{1}} \le R_{F,100\bar{1}} \quad \mbox{for all}\ k \ge 4, \\
R_{F,100001} \le \widetilde{R}_{F,1000\bar{1}},\ R_{F,10^k1} \le \widetilde{R}_{F,100\bar{1}} \quad \mbox{for}\ k = 3\ \mbox{and all}\ k \ge 5.
\end{gather*}
Here, $\widetilde{R}_{F,y}$ denotes the matrix which is obtained from $R_{F,y}$ by exchanging each column with its symmetric counterpart.  
Since $R_{F,y} \mathbf{w} = R_{F,y0} \mathbf{w}$ for all $y \in \Sigma^*$, we obtain that $R_{F,y} \mathbf{w}$ is maximal for $y$ with period $100\bar{1}00$.
For $y = (100\bar{1}00)^{k/6}$, where a fractional power $(z_1\cdots z_6)^{k/6}$ denotes as usual the word $(z_1 \cdots z_6)^{\lfloor k/6\rfloor}\, z_1 \cdots z_{k-6\lfloor k/6\rfloor}$, we have $\mathbf{v} R_{F,y} \mathbf{w} = 2^{\lfloor(k-1)/3\rfloor}$, thus~\eqref{e:maxvRw} holds.
By Corollary~\ref{c:transducer}, this means that $f_k(n) = 2^{\lfloor(k-1)/3\rfloor}$ for $n \sim_F (100\bar{1}00)^{k/6}$, and $f_k(n) \le 2^{\lfloor(k-1)/3\rfloor}$ for all $n$ with greedy $F$-expansion of minimal weight of length at most~$k$.

It remains to consider $f_k(n)$ for those $n$ whose greedy $F$-expansion of minimal weight (without leading zeros) is longer than~$k$.
The transducer $\mathcal{N}_{F,\{-1,0,1\}}$ shows that any $F$-expansion of minimal weight in $\{-1,0,1\}^*$ is at most $1$ shorter than the corresponding greedy $F$-expansion of minimal weight.
Thus we have to consider $n \sim_F y_k y_{k-1} \cdots y_0 \in G_F$, where we assume w.l.o.g.\ $y_k = 1$.
For such an~$n$, we have $f_k(n) = \mathbf{v}' R_{F,y_{k-1} \cdots y_0} \mathbf{w}$, where $\mathbf{v}'$ is the indicator vector of the state which is reached from the initial state by the transition labeled by $(1,0)$.
Now, equations and inequalities such as $\mathbf{v}' R_{F,00\bar{1}001} = \mathbf{v} R_{F,00\bar{1}001}$ and $\mathbf{v}' R_{F,000\bar{1}} \le \mathbf{v} R_{F,\bar{1}}$ show that $f_k(n) \le \max_{y\in\Sigma^k} \mathbf{v} R_{F,y} \mathbf{w} = 2^{\lfloor(k-1)/3\rfloor}$.
\end{proof}

\bibliography{redundant}
\bibliographystyle{amsalpha}
\end{document}